\documentclass[12pt]{article}

\usepackage{booktabs}

\usepackage{graphicx}
\usepackage{nameref}
\usepackage[shortlabels,inline]{enumitem}
\usepackage{empheq}
\usepackage[shortlabels]{enumitem}
\setlist[enumerate]{nosep}
\usepackage[doc]{optional}
\usepackage{xcolor}
\usepackage[colorlinks=true,
            linkcolor=refkey,
            urlcolor=lblue,
            citecolor=red]{hyperref}
\usepackage{float}
\usepackage{soul}
\usepackage{graphicx}
\definecolor{labelkey}{rgb}{0,0.08,0.45}
\definecolor{refkey}{rgb}{0,0.6,0.0}
\definecolor{Brown}{rgb}{0.45,0.0,0.05}
\definecolor{lime}{rgb}{0.00,0.8,0.0}
\definecolor{lblue}{rgb}{0.5,0.5,0.99}

 \usepackage{mathpazo}

\colorlet{hlcyan}{cyan!30}

\usepackage{stmaryrd}
\usepackage{amssymb}
\makeatletter
\def\namedlabel#1#2{\begingroup
   \def\@currentlabel{#2}%
   \label{#1}\endgroup
}
\makeatother

\newcommand{\seppfour}{\setlength{\itemsep}{-4pt}}

\newcommand{\bx}{\ensuremath{\mathbf{x}}}
\newcommand{\bzero}{\ensuremath{{\boldsymbol{0}}}}

\newcommand{\weakly}{\ensuremath{\:{\rightharpoonup}\:}}

\newcommand{\nnn}{\ensuremath{{n\in{\mathbb N}}}}
\newcommand{\thalb}{\ensuremath{\tfrac{1}{2}}}
\newcommand{\menge}[2]{\big\{{#1}~\big |~{#2}\big\}}

\newcommand{\fenv}[1]%
{\ensuremath{\,\overrightarrow{\operatorname{env}}_{#1}}}
\newcommand{\benv}[1]%
{\ensuremath{\,\overleftarrow{\operatorname{env}}_{#1}}}

\newcommand{\scal}[2]{\left\langle{#1},{#2}  \right\rangle}
\newcommand{\tscal}[2]{\langle{#1},{#2}\rangle}

\newcommand{\exi}{\ensuremath{\exists\,}}
\newcommand{\zeroun}{\ensuremath{\left]0,1\right[}}
\newcommand{\RR}{\ensuremath{\mathbb R}}
\newcommand{\RP}{\ensuremath{\mathbb{R}_+}}

\newcommand{\RM}{\ensuremath{\mathbb{R}_-}}

\newcommand{\NN}{\ensuremath{\mathbb N}}

\newcommand{\dom}{\ensuremath{\operatorname{dom}}}

\newcommand{\bd}{\ensuremath{\mathbf{d}}}
\newcommand{\ran}{\ensuremath{{\operatorname{ran}}\,}}

\newcommand{\Fix}{\ensuremath{\operatorname{Fix}}}

\newcommand{\Id}{\ensuremath{\operatorname{Id}}}

\newcommand{\bX}{\ensuremath{{\mathbf{X}}}}

\newcommand{\bC}{\ensuremath{{\mathbf{C}}}}
\newcommand{\bQ}{\ensuremath{{\mathbf{Q}}}}
\newcommand{\bP}{\ensuremath{{\mathbf{P}}}}
\newcommand{\bR}{\ensuremath{{\mathbf{R}}}}
\newcommand{\bA}{\ensuremath{{\mathbf{A}}}}

\newcommand{\bT}{\ensuremath{{\mathbf{T}}}}
\newcommand{\bZ}{\ensuremath{{\mathbf{Z}}}}

\newcommand{\bD}{\ensuremath{{\boldsymbol{\Delta}}}}

\newcommand{\bB}{\ensuremath{{\mathbf{B}}}}
\newcommand{\bc}{\ensuremath{\mathbf{c}}}
\newcommand{\by}{\ensuremath{\mathbf{y}}}
\newcommand{\be}{\ensuremath{\mathbf{e}}}
\newcommand{\bz}{\ensuremath{\mathbf{z}}}

{\begin{list}{}{%
\settowidth{\labelwidth}{\textrm{#1~}}%
\setlength{\leftmargin}{\labelwidth+\labelsep}}}
{\end{list}}
\usepackage{amsthm}
\usepackage[capitalize,nameinlink]{cleveref}
\crefname{lemma}{lemma}{lemmas}
\crefname{equation}{}{equations}
\crefname{figure}{Figure}{Figures}
\crefname{chapter}{Appendix}{chapters}
\crefname{item}{}{items}
\crefname{enumi}{}{}
\newtheorem{theorem}{Theorem}[section]

\newtheorem{corollary}[theorem]{Corollary}

\newtheorem{proposition}[theorem]{Proposition}

\newtheorem{ex}[theorem]{Example}
\newtheorem{fact}[theorem]{Fact}
\newtheorem{remark}[theorem]{Remark}

\newcommand{\bv}{\ensuremath{{\mathbf{v}}}}

\providecommand{\RR}{\mathbb{R}}

\providecommand{\ran}{\operatorname{ran}}

\providecommand{\dom}{\operatorname{dom}}

\providecommand{\epi}{\operatorname{epi}}

\providecommand{\gra}{\operatorname{gra}}
\providecommand{\Id}{\operatorname{{ Id}}}

\providecommand{\NN}{\mathbb{N}}

\providecommand{\ran}{\operatorname{ran}}
\providecommand{\rec}{\operatorname{rec}}
\providecommand{\Id}{\operatorname{Id}}

\providecommand{\RR}{\mathbb{R}}
\providecommand{\NN}{\mathbb{N}}

\definecolor{myblue}{rgb}{.8, .8, 1}
  \newcommand*\mybluebox[1]{%
    \colorbox{myblue}{\hspace{1em}#1\hspace{1em}}}

\allowdisplaybreaks 

\begin{document}

\title{\textsc{
The difference vectors for convex sets and a resolution of the geometry conjecture}}
\author{
Salihah Alwadani\thanks{
Mathematics, University
of British Columbia,
Kelowna, B.C.\ V1V~1V7, Canada. E-mail:
\texttt{saliha01@mail.ubc.ca}.},~ 
Heinz H.\ Bauschke\thanks{
Mathematics, University
of British Columbia,
Kelowna, B.C.\ V1V~1V7, Canada. E-mail:
\texttt{heinz.bauschke@ubc.ca}.},~
Julian P.\ Revalski\thanks{
Institute of Mathematics and Informatics,
Bulgarian Academy of Sciences, 
Acad.\ G.\ Bonchev str., Block~8, 
1113~Sofia, Bulgaria. E-mail: 
\texttt{revalski@math.bas.bg}.},
~and~
Xianfu Wang\thanks{
Mathematics, University
of British Columbia,
Kelowna, B.C.\ V1V~1V7, Canada. E-mail:
\texttt{shawn.wang@ubc.ca}.}
}

\date{May 28, 2021 (version accepted for publication)}
\maketitle

\vskip 8mm

\begin{abstract} 
The geometry conjecture, which  was posed nearly 
a quarter of a century ago, states 
that the fixed point
set of the composition of projectors onto nonempty closed convex 
sets in Hilbert space is actually equal to the intersection of 
certain translations of the underlying sets. 

In this paper, we provide a complete resolution of 
the geometry conjecture. Our proof relies on 
monotone operator theory. We revisit previously known results
and provide various illustrative examples.
Comments on the numerical computation of the quantities involved 
are also presented. 
\end{abstract}

{\small

}

\section{Introduction}

\subsection{Fixed points of compositions of projectors}
Throughout,
\begin{empheq}[box=\mybluebox]{equation}
  \text{$X$ is a real Hilbert space}
\end{empheq}
and
\begin{empheq}[box=\mybluebox]{equation}
\text{$C_1,\ldots,C_m$ are nonempty closed convex subsets of $X$,}
\end{empheq}
with projectors 
$P_{C_1},\ldots,P_{C_m}$ which we also write more simply as 
$P_1,\ldots,P_m$, 
and with $m\in\{2,3,\ldots,\}$. 
We define the fixed point sets of the cyclic compositions by 
\begin{empheq}[box=\mybluebox]{equation}
\label{e:defF_i}
  F_m := \Fix(P_m\cdots P_1), 
  F_{m-1} := \Fix(P_{m-1}\cdots P_1P_m),
  \ldots,
  F_1 := \Fix(P_1P_m\cdots P_2).
\end{empheq}
Compositions of projectors are often employed in projection methods.
This is a vast area which we will not summarize here; however, we 
refer the reader to \cite{CenZak} as a starting point as well as 
the very recent paper \cite{ComPes20}. 

\subsection{The geometry conjecture, difference vectors, and cycles}
\label{ss:geocon}

The \textbf{geometry conjecture}, formulated first in 1997
(see \cite[Conjecture~5.1.6]{BBL}), states that 
there exists a list of vectors
$v_1,\ldots,v_m$ in $X$ such that 
\begin{equation}
\label{e:perp}
v_1+v_2+\cdots+v_m = 0
\end{equation}
and 
\begin{equation}
\label{e:Fm}
F_m = C_m \cap (C_{m-1}+v_{m-1}) \cap \cdots \cap 
(C_1+v_{1}+\cdots+v_{m-1}),
\end{equation}
and analogously for $F_{m-1},\ldots,F_1$. 
These vectors form the tuple $(v_1,\ldots,v_m)$ of \textbf{difference vectors} 
and they are sometimes also referred to as displacement vectors or gap vectors. 

This conjecture is known to be true when 
$m=2$ or $C_1\cap C_2 \cap \cdots \cap C_m\neq\varnothing$;
see \cite[Subsection~5.1]{BBL}. 
The following is known when  all sets $F_i$ are nonempty: 
let $f_1\in F_1$, and set $f_2 := P_2f_1, f_3 := P_3f_2, \ldots,
f_{m} := P_{m}f_{m-1}$; 
we shall refer to the tuple $(f_1,\ldots,f_m)$ as a \textbf{cycle}. 
(Cycles are of interest even when $C_1,\ldots,C_m$ are all hyperplanes 
--- see \cite[Chapter~8]{ByrneSP} and \cite[Chapter~50]{ByrneAIM}.)
Setting 
\begin{equation}
\label{e:cheatcycle}
  v_1 = f_2-f_1,\; v_2 = f_3-f_2,\; \ldots,\; v_{m-1} = f_m-f_{m-1},\; v_m = f_1-f_m, 
\end{equation}
which turns out to be independent of the cycle chosen, 
makes \cref{e:perp} true and yields 
``one half'' of \cref{e:Fm}, namely: 
$F_m \subseteq C_m \cap (C_{m-1}+v_{m-1}) \cap \cdots \cap 
(C_1+v_{1}+\cdots+v_{m-1})$ 
and analogously for $F_{m-1},\ldots,F_1$
(see \cite[Theorem~5.1.2]{BBL}).
However, this description is not fully satisfying --- 
it is only \emph{implicit} in the sense 
it was not known 
what the difference vectors are when the fixed point sets $F_i$ are empty. 
The sole exception to this mystery was the case when $m=2$ which allowed 
for the \emph{explicit} description of the 
two difference vectors by 
\begin{equation}
\label{e:200805d}
  P_{\overline{C_2-C_1}}(0), P_{\overline{C_1-C_2}}(0); 
\end{equation}
see \cite[Lemmas~2.1 and 2.3]{02.pdf}. Note 
that this description is \emph{not} based on the fixed point sets $F_1,F_2$.  
In this particular case, these fixed point set have the beautiful 
description (see \cite[Theorem~2]{CG})
\begin{equation}
  F_1 = \menge{x\in C_1}{d_{C_2}(x)=\inf\|C_1-C_2\|},
  F_2 = \menge{x\in C_2}{d_{C_1}(x)=\inf\|C_1-C_2\|};
\end{equation}
moreover, the cycles $(f_1,f_2)$ are precisely the minimizers 
of the bivariate function 
\begin{equation}
  X\times X\to\RR\colon (x_1,x_2)\mapsto \|x_1-x_2\| + \iota_{C_1}(x_1)+\iota_{C_2}(x_2),
\end{equation}
where $d_S$ and $\iota_{S}$ denote the distance and indicator function of a subset $S$ of $X$, 
respectively.  
(See \cite{CG}, \cite{02.pdf}, \cite{01.pdf}, and \cite{BBL} for much more on the case 
when $m=2$.)
A referee also pointed out that when $m=2$ and $F_1=F_2=\varnothing$ one cannot expect 
uniqueness of the difference vectors as one may simply separate the sets even further.

The case when $m\geq 3$ is very interesting:
The \emph{negative} result of 
Baillon, Combettes, and Cominetti (see \cite[Theorem~2.3]{BCC}) 
states that when $X$ is at least two-dimensional, then 
there is \emph{no} function $\varphi$ such that 
the cycles are precisely the minimizers of the function 
$\varphi(x_1,\ldots,x_m)+\iota_{C_1}(x_1)+\cdots + \iota_{C_m}(x_m)$.
(When $m=2$, we can pick $\varphi(x_1,x_2)=\|x_1-x_2\|$ or even the 
differentiable function $\varphi(x_1,x_2)=\thalb\|x_1-x_2\|^2$. 
For results on underrelaxed projectors, see also \cite{BCC+} and \cite{CRW}.)
Even when cycles exist, the ``meaning'' of the distance vector was
not understood. 

\subsection{Aim and outline of this paper}

\emph{The aim of this paper is to settle the geometry 
conjecture in the affirmative.} 
The resolution depends on key results from monotone operator theory and yields 
a formula for the difference vectors. 

The remainder of this paper is organized as follows. 
In \cref{s:aux}, we reformulate cycles and difference vectors 
in a product space using Attouch-Th\'era duality. 
The proof of the geometry conjecture is then presented in \cref{s:main} 
(see \cref{t:main}). The cases $m=2,m=3$ are investigated in 
\cref{s:m=2,s:m=3}. Numerical considerations are presented in 
\cref{s:Banach,s:fb}. The paper concludes with a summary 
and perspectives for future work in \cref{s:done}.

Notation is largely from \cite{BC2017} to which we also refer for 
background material on projections, convex analysis, and monotone 
operator theory. For valuable references on monotone operator theory 
see, e.g., \cite{Brezis}, \cite{BurIus}, \cite{Simons1}, and \cite{Simons2}.

\section{The displacement of the circular right shift operator}

\label{s:aux}

\subsection{Product space and Attouch-Th\'era duality}

From now on, we will also work in the product space 
\begin{empheq}[box=\mybluebox]{equation}
  \bX := X^m
\end{empheq}
in which we set 
\begin{empheq}[box=\mybluebox]{equation}
  \bC := C_1\times \cdots \times C_m
\;\;\text{and}\;\;
  \bD := \menge{(x,\ldots,x)\in \bX}{x\in X}.
\end{empheq}
It is well known that the projectors onto these sets are given by 
\begin{equation}
  P_{\bC}(x_1,\ldots,x_m) = \big(P_1x_1,\ldots,P_mx_m\big)
\end{equation}
and 
\begin{equation}
  P_{\bD}(x_1,\ldots,x_m) = \frac{1}{m}\Big(\sum_{i=1}^mx_i,\ldots,\sum_{i=1}^mx_i\Big)
\end{equation}
respectively (see, e.g., \cite[Proposition~29.4 and Proposition~26.4(iii)]{BC2017}). 
Next, we define the circular right-shift operator 
\begin{empheq}[box=\mybluebox]{equation}
\label{e:defofR}
  \bR \colon \bX\to\bX\colon (x_1,x_2,\ldots,x_m)
  \mapsto (x_m,x_1,x_2,\ldots,x_{m-1}).
\end{empheq}
Recall (see \cref{ss:geocon}) that 
$\bz = (z_1,\ldots,z_m)\in\bX$ 
is a cycle 
if $z_1=P_1z_m$, $z_2=P_2z_1$, 
\ldots, $z_m = P_mz_{m-1}$, which 
can be elegantly reformulated in $\bX$ as the fixed point equation 
\begin{equation}
\label{e:200705c}
  \bz = P_{\bC}(\bR\bz).
\end{equation}
Denote the (possibly empty) \textbf{set of all cycles} by 
\begin{empheq}[box=\mybluebox]{equation}
  \bZ := \Fix(P_\bC\bR).
\end{empheq}
In passing, we note that if $Q_i\colon (x_1,\ldots,x_m)\to x_i$, 
then $F_i = Q_i(\bZ)$. 
Because $P_\bC = (\Id+N_\bC)^{-1}$, where $N_\bC$ denotes the normal cone operator of $\bC$, 
it follows that 
\cref{e:200705c} is equivalent to 
$\bR\bz\in(\Id+N_\bC)(\bz)$ and to 
\begin{equation}
\label{e:ATprimal}
0\in N_{\bC}(\bz) + (\Id-\bR)(\bz).
\end{equation}
We view this last inclusion sum problem
as primal (Attouch-Th\'era) problem 
for the pair $(N_\bC,\Id-\bR)$. 
(See \cite{AtTh} and \cite{BBHM} for more on Attouch-Th\'era duality.)
In view of the linearity of $\bR$, 
the Attouch-Th\'era dual problem simplifies to 
\begin{equation}
\label{e:ATdual}
0 \in N_{\bC}^{-1}(\by) + (\Id-\bR)^{-1}(\by).
\end{equation}
If $\bz$ is any cycle; equivalently, a solution to the primal problem 
\cref{e:ATprimal}, 
then a direct computation (or \cite[Proposition~2.4(iii)]{BBHM}) shows that 
$N_{\bC}(\bz)\cap -(\Id-\bR)(\bz)$ is a nonempty subset of dual solutions. 
Even better, both $N_{\bC}$ and $\Id-\bR$ are \emph{paramonotone} 
in the sense of Iusem \cite{Iusem}
by, e.g., \cite[Example~22.4(i) and Example~22.9]{BC2017}. 
It thus follows from \cite[Theorem~5.3]{BBHM} that 
\begin{equation}
\label{e:cycletodual}
(\forall \bz\in\bZ)\quad 
\bR\bz-\bz \text{~is the \emph{unique} solution of \cref{e:ATdual}}
\end{equation}
and that 
\begin{equation}
\label{e:dualtocycles}
\text{if $\by$ solves \cref{e:ATdual}, then~~}
\bZ =  N_\bC^{-1}(\by)\cap -(\Id-\bR)^{-1}(\by)\neq \varnothing
\end{equation}

\subsection{$(\Id-\bR)^{-1}$ and the skew operator $\bT$}

Recall the definition of the circular right shift operator $\bR$ 
(see \cref{e:defofR}). 
By \cite[Proposition~2.4]{ABRW}, we have 
\begin{equation}
\label{e:JulianPbD}
  P_\bD = \frac{1}{m}\sum_{k=0}^{m-1}\bR^k.
  \end{equation}
Now define 
\begin{empheq}[box=\mybluebox]{equation}
\label{e:defofT}
\bT = \frac{1}{2m}\sum_{k=1}^{m-1}(m-2k)\bR^k, 
\end{empheq}
which is a \emph{skew} (hence maximally monotone) linear operator on $\bX$, i.e., 
\begin{equation}
\label{e:Tskew}
\bT^*=-\bT
\end{equation} 
with 
\begin{equation}
\label{e:ranT}
\ran\bT\subseteq \bD^\perp
\end{equation}
(see \cite[Proposition~3.2(ii)\&(iii)]{ABRW}). 
Then 
\cite[Theorem~3.3]{ABRW} states that 
\begin{equation}
\label{e:Id-Rinv}
(\Id-\bR)^{-1} = \thalb \Id + N_{\bD^\perp}+\bT. 
\end{equation}
This form of $(\Id-\bR)^{-1}$ makes it clear that this operator 
is \emph{strongly monotone} with constant $\thalb$ which implies that 
\begin{equation}
\text{the dual problem \cref{e:ATdual} has \emph{at most} one solution}
\end{equation}
which is consistent with \cref{e:cycletodual}.
(A feature of Attouch-Th\'era duality is that either both primal 
and dual have solutions or they both don't. 
It is possible that there is no cycle and hence no dual solution;
see \cref{sss:nocycles}.) 

We now collect some useful identities.

\begin{proposition}
\label{p:200715a}
We have $P_\bD\bR = \bR P_\bD = P_\bD$ and 
hence $P_{\bD^\perp}\bR = \bR - P_\bD$.
\end{proposition}
\begin{proof}
Recalling \cref{e:JulianPbD}, we observe that 
$P_\bD \bR = \bR P_\bD$. 
Furthermore, 
\begin{align}
\bR P_\bD = \frac{1}{m}\sum_{k=0}^{m-1}\bR^{k+1}
= \frac{1}{m}\sum_{k=1}^{m}\bR^{k}
= \frac{1}{m}\sum_{k=0}^{m-1}\bR^{k}
= P_\bD
\end{align}
because $\bR^m=\bR^0=\Id$.
It follows that 
$P_{\bD^\perp}\bR = 
(\Id-P_\bD)\bR = \bR - P_\bD$ .
\end{proof}

For the remainder of this section, let us abbreviate
\begin{empheq}[box=\mybluebox]{equation}
\label{e:localQ}
\bQ := \frac{1}{m} \sum_{k=1}^{m-1}k \bR^k.
\end{empheq}
Clearly, $\bQ$ commutes with $\bR$, and hence also with $P_\bD$ 
by \cref{e:JulianPbD}. 

\begin{proposition}
\label{p:200715b}
$2\bQ P_\bD = (m-1)P_\bD$. 
\end{proposition}
\begin{proof}
Using \cref{p:200715a}, we see that 
\begin{align}
\bQ P_\bD
&= \frac{1}{m} \sum_{k=1}^{m-1}k \bR^k P_\bD 
= \frac{1}{m} \sum_{k=1}^{m-1}k P_\bD 
= \frac{1}{m}\frac{(m-1)m}{2} P_\bD 
= \frac{m-1}{2} P_\bD
\end{align}
as claimed.
\end{proof}

\begin{proposition}
\label{p:200715c}
$-\bQ(\Id-\bR) = P_{\bD^\perp}$. 
\end{proposition}
\begin{proof}
Using 
\cref{e:localQ} and \cref{e:JulianPbD}, we obtain 
\begin{subequations}
\begin{align}
-m\bQ(\Id-\bR)
&=
(\bR-\Id)(m\bQ)
= 
(\bR-\Id)\sum_{k=1}^{m-1}k\bR^k\\
&= 
\sum_{k=1}^{m-1}k\bR^{k+1} - \sum_{k=1}^{m-1}k\bR^k
= 
\sum_{k=2}^{m}(k-1)\bR^{k} - \sum_{k=1}^{m-1}k\bR^k\\
&= 
(m-1)\bR^m + \Big( \sum_{k=2}^{m-1}\big((k-1)-k\big)\bR^k\Big) - \bR
\\ 
&= 
(m-1)\Id - \Big( \sum_{k=2}^{m-1}\bR^k\Big) -\bR\\
&= 
m\Id - \sum_{k=0}^{m-1}\bR^k
= m\Id - m P_\bD
= m P_{\bD^\perp},
\end{align}
\end{subequations}
which completes the proof. 
\end{proof}

We are now ready for the main result of this section
which will play a key role in subsequent sections. 

\begin{theorem}
We have 
\begin{equation}
\label{e:200715a}
  \thalb\Id + \bT = \tfrac{m}{2}P_\bD - \bQ
\end{equation}
and 
\begin{equation}
\label{e:200715b}
\big(\thalb\Id + \bT\big)^{-1}
= \Id-\bR+2P_\bD. 
\end{equation}
\end{theorem}
\begin{proof}
Using 
\cref{e:defofT},
\cref{e:JulianPbD}, and 
\cref{e:localQ},
we have 
\begin{align}
\bT 
&= \frac{1}{2m}\sum_{k=1}^{m-1}(m-2k)\bR^k
= \frac{1}{2}\sum_{k=1}^{m-1}\bR^k 
- \frac{1}{m}\sum_{k=1}^{m-1}k\bR^k
=
\frac{1}{2}\Big(-\Id + m P_\bD\Big) - \bQ
\end{align}
which gives 
\cref{e:200715a}.

Next, using \cref{e:200715a},
\cref{p:200715a}, 
\cref{p:200715b}, 
and \cref{p:200715c}, 
we obtain 
\begin{subequations}
\begin{align}
\big(\thalb\Id + \bT \big)\big(\Id-\bR+2P_\bD)
&= \big(\tfrac{m}{2}P_\bD-\bQ\big)\big(\Id-\bR+2P_\bD)\\
&= \tfrac{m}{2}\big(P_\bD-P_\bD\bR\big) + m P_\bD - \bQ(\Id-\bR) -2\bQ P_\bD
\\ 
&= mP_\bD + P_{\bD^\perp} - (m-1)P_\bD
= P_\bD + P_{\bD^\perp}\\
&= \Id.
\end{align}
\end{subequations}
This verifies \cref{e:200715b} and thus completes the proof.
\end{proof}

\begin{corollary}
We have 
\begin{equation}
\label{e:200705e}
\big(\thalb\Id+T\big)^{-1}|_{\bD^\perp} = (\Id-\bR)|_{\bD^\perp},
\end{equation}
\end{corollary}
\begin{proof}
From \cref{e:200715b}, we have
$\big(\thalb\Id+T\big)^{-1}|_{\bD^\perp} = 
(\Id-\bR+2P_\bD)|_{\bD^\perp}
= (\Id-\bR)|_{\bD^\perp}$. 
\end{proof}

\section{The proof of the geometry conjecture}

\label{s:main}

Armed with \cref{e:Id-Rinv}, 
write the operator from the dual problem \cref{e:ATdual}
as 
\begin{equation}
\label{e:200804a}
  N_\bC^{-1} + (\Id-\bR)^{-1}
  = N_\bC^{-1} + \thalb\Id+N_{\bD^\perp}+\bT.
\end{equation}
This operator is in general not maximally monotone. 
On the other hand,
$N_\bC^{-1} + N_{\bD^\perp}
= N_{\bC}^{-1} + N_{\bD}^{-1}
= \partial \sigma_\bC + \partial \sigma_\bD
\subseteq \partial \sigma_{\bC+\bD}$, 
where $\sigma_{\mathbf{S}}$ denotes the support function of a subset 
$\mathbf{S}$ of $\bX$. 
Altogether, instead of working with \cref{e:200804a}, which has 
no solution if there are no cycles, 
we propose to work with the 
the \emph{enlarged} dual problem featuring the \emph{maximally 
and strongly monotone} operator 
\begin{equation}
  \thalb\Id + \bT + \partial\sigma_{\bC+\bD}. 
\end{equation}
Using, e.g., \cite[Proposition~22.11(ii)]{BC2017},   
the corresponding inclusion problem 
always has a unique zero, which we denote by 
$\by\in\bX$: 
\begin{empheq}[box=\mybluebox]{equation}
\label{e:200705a}
0 \in \thalb \by + \bT\by + \partial\sigma_{\bC+\bD}(\by).
\end{empheq}
(In fact, $\by$ is the resolvent of the maximally monotone 
operator $2\bT +2\partial\sigma_{\bC+\bD}$, 
evaluated at $0$.) 
Note that, using \cite[Proposition~6.49 and Example~11.2]{BC2017}
$\by \in \dom\partial\sigma_{\bC+\bD} 
\subseteq \dom\sigma_{\bC+\bD}
=\dom(\sigma_{\bC}+\sigma_{\bD})
=\dom\sigma_\bC \cap \dom \sigma_\bD
\subseteq (\rec\bC)^\ominus \cap (\rec\bD)^\ominus$;
thus, 
\begin{equation}
\label{e:yinDperp}
  \by \in (\rec \bC)^\ominus \cap \bD^\perp.
\end{equation}
Now define 
\begin{empheq}[box=\mybluebox]{equation}
\label{e:defofe}
\be := -\thalb \by -\bT\by \in \bD^\perp,
\end{empheq}
where $\be \in \bD^\perp$ 
because 
$\by\in\bD^\perp$ (see \cref{e:yinDperp}) 
and 
$\ran\bT\subseteq \bD^\perp$ (see \cref{e:ranT}).
Note that $-\be = (\thalb\Id+\bT)\by$. 
Hence \cref{e:200705e} yields
\begin{equation}
\label{e:200705f}
\by = (\Id-\bR)(-\be) = \bR\be-\be.
\end{equation}
Note that \cref{e:200705a} is equivalent to 
$\be \in \partial\sigma_{\bC+\bD}(\by) = \partial \iota^*_{\bC+\bD}(\by)$, 
and hence also to  
\begin{equation}
\label{e:200705b}
\by \in N_{\overline{\bC+\bD}}(\be),
\end{equation}
where the superscript ``$\mbox{}^*$'' denotes Fenchel conjugation. 
We pause here to record the following result which 
provides a certificate for $\by$:

\begin{proposition}{\rm \textbf{(a characterization of $\by$)}}
\label{p:verifyby}
The unique solution to \cref{e:200705a} is the unique vector $\by$ satisfying 
the following:
\begin{equation}
\label{e:verifyby}
\by\in\bD^\perp,\;\;
-\thalb\by-\bT\by\in\overline{\bC+\bD},\;\;
\text{and}\;\;
(\forall \bc\in\bC)\;
\scal{\bc}{\by}\leq -\thalb\|\by\|^2.
\end{equation}
\end{proposition}
\begin{proof}
As seen, $\by$ solves \cref{e:200705a}
if and only if \cref{e:200705b} holds with $\be$ defined in \cref{e:defofe}. 
The latter condition is equivalent to 
$\be = -\thalb\by-\bT\by\in\overline{\bC+\bD}$ and 
$(\forall (\bc,\bd)\in\bC\times\bD)$
$\tscal{\by}{\bc+\bd+\thalb\by+\bT\by}\leq 0$. 
Because $\by\in\bD^\perp$ (see \cref{e:yinDperp}) 
and $\bT$ is skew (see \cref{e:Tskew}),
the last condition is indeed equivalent to \cref{e:verifyby}.
\end{proof}

Combining \cref{e:defofe} and \cref{e:200705b}, 
we deduce that 
\begin{equation}
\label{e:200723c}
  \be \in \bD^\perp \cap \overline{\bC+\bD}. 
\end{equation}
(This last intersection $\bD^\perp\cap \overline{\bC+\bD}$ 
need not be a singleton as we can 
see by studying the case when $C_1=C_2=\cdots=C_m=X$ and hence $\bC=\bX$, 
in which case the intersection is $\bD^\perp$.)

\begin{theorem}
With $\by$ and $\be$ as defined in 
\cref{e:200705a} and \cref{e:defofe} respectively, 
the set of cycles is given by 
\begin{subequations}
\begin{align}
\bZ 
&= N_\bC^{-1}(\by)\cap(\be +\bD) \label{e:main0}\\
&= \be + \big(\bD \cap (\bC-\be) \big). \label{e:main1}
\end{align}
\end{subequations}
\end{theorem}
\begin{proof}
First, 
$\bC \subseteq \overline{\bC+\bD}$
because $0\in\bD$. 
Hence 
$(\forall \bc\in\bC)$
$\scal{\by}{\bc-\be}\leq 0$ by \cref{e:200705b}.
It follows that 
\begin{equation}
\label{e:200705d}
\sigma_{\bC}(\by) \leq 
\scal{\by}{\be} = -\thalb\|\by\|^2,
\end{equation}
where the equality follows from \cref{e:Tskew}
and \cref{e:defofe}. 

Next, $\by$ might even solve the original dual \cref{e:ATdual} in which case 
$\bZ$ is given by \cref{e:dualtocycles}. Whether or not this is the case, we 
\emph{always} have, using 
\cref{e:Id-Rinv}, \cref{e:defofe}, and 
\cref{e:yinDperp}, 
\begin{subequations}
\begin{align}
N_\bC^{-1}(\by)\cap -(\Id-\bR)^{-1}(\by)
& = N_\bC^{-1}(\by)\cap \big(-\thalb\by - \bT\by - N_{\bD^\perp}(\by)\big)\\
&= N_{\bC}^{-1}(\by) \cap (\be + \bD). 
\end{align}
\end{subequations}
Altogether, this yields \cref{e:main0}. 

Now let $\bx \in \bX$ and set $\bd := \bx-\be$.
Then, using \cref{e:main0}, \cref{e:yinDperp}, 
and \cref{e:200705d}, 
we have the equivalences 
\begin{subequations}
\begin{align}
\bx \in \bZ
&\Leftrightarrow 
\bx \in N_\bC^{-1}(\by) \cap (\be+\bD)\\
&\Leftrightarrow 
\by \in N_\bC(\bx) \text{~and~} \bx-\be\in\bD\\
&\Leftrightarrow 
\by \in N_\bC(\bd+\be) \text{~and~} \bd\in\bD\\
&\Leftrightarrow 
\bd\in\bD,\, \bd+\be \in\bC,
\text{~and~} (\forall \bc\in\bC)\;\;\scal{\by}{\bc-(\bd+\be)}\leq 0\\
&\Leftrightarrow 
\bd\in\bD,\, \bd+\be \in\bC,
\text{~and~} (\forall \bc\in\bC)\;\;\scal{\by}{\bc-\be}\leq 0\\
&\Leftrightarrow 
\bd \in \bD\cap (\bC-\be)
\text{~and~}
\sigma_\bC(\by) \leq \scal{\by}{\be}\\
&\Leftrightarrow 
\bx-\be \in \bD\cap (\bC-\be),
\end{align}
\end{subequations}
and this gives \cref{e:main1}. 
\end{proof}

\begin{corollary}
\label{c:Shawn}
The following hold:
\begin{enumerate}
\item 
\label{c:Shawn0}
{\rm \textbf{(orthogonal decomposition of $\bZ$)}}
$P_{\bD^\perp}(\bZ) \subseteq \{\be\}$ and 
$P_{\bD}(\bZ) = \bD \cap (\bC-\be)$. 
\item 
\label{c:Shawn1}
$\bZ \neq\varnothing
\Leftrightarrow
\be \in \bC+\bD$.
\item 
\label{c:Shawn2}
If $\be \in \bC+\bD$, say 
$\be = \bc+\bd$, where $\bc\in\bC$ and $\bd\in\bD$,
then $\bc \in\bZ$.
\item 
\label{c:Shawn3}
If $\bz\in\bZ$, then $\be = P_{\bD^\perp}\bz 
\in (\bC+\bD)\cap\bD^\perp$. 
\end{enumerate}
\end{corollary}
\begin{proof}
\cref{c:Shawn0}: 
Recall that $\be\in\bD^\perp$ by \cref{e:defofe}. 
Clearly, $\bD\cap (\bC-\be)\subseteq \bD$.
Using 
\cref{e:main1}, we obtain an orthogonal decomposition of $\bZ$,
with $\bD^\perp$ component $P_{\bD^\perp}(\bZ)\subseteq \{\be\}$
and $P_{\bD}(\bZ)=\bD\cap (\bC-\be)$. 
\cref{c:Shawn1}: 
Indeed, using \cref{e:main1}, we have 
$\bZ\neq\varnothing$
$\Leftrightarrow$
$\bD \cap (\bC-\be)\neq\varnothing$
$\Leftrightarrow$
$(\exi \bc\in\bC)$ $\bc-\be\in\bD$ 
$\Leftrightarrow$
$\be \in \bC+\bD$.
\cref{c:Shawn2}: 
Indeed,
$\bc-\be = -\bd \in (\bD\cap(\bC-\be))$ and so 
$\bc\in\be+(\bD\cap(\bC-\be))=\bZ$ using \cref{e:main1}. 
\cref{c:Shawn3}:
Using \cref{c:Shawn0} and \cref{e:defofe},
we obtain 
$\be= P_{\bD^\perp}\bz 
= \bz - P_{\bD}\bz \in (\bZ-\bD)\cap \bD^\perp 
\subseteq (\bC+\bD)\cap\bD^\perp$. 
\end{proof}

At long last, we define 
\begin{empheq}[box=\mybluebox]{equation}
\label{e:defbv}
  \bv := \bR^*\be-\be \in \bD^\perp, 
\end{empheq}
where $\bv \in\bD^\perp$ 
because $\ran(\Id-\bR^*) 
= \ran(\Id-\bR)=\bD^\perp$ by 
\cite[Theorem~2.2(iv)]{Victoria}. 
Also \cref{e:200705f} yields
\begin{equation}
\bv = \bR^*\be-\be = -\bR^*(\bR\be-\be)=-\bR^*\by,
\end{equation}
which in turn gives 
\begin{equation}
\label{e:yfromv}
\by = -\bR\bv.
\end{equation}
Because $\bR^*$ is the circular \emph{left} shift, 
\cref{e:defbv} and \cref{e:yfromv} yield   
\begin{subequations}
\label{e:anotherbv}
\begin{align}
\bv 
&= (e_2-e_1,e_3-e_2,\ldots,e_m-e_{m-1},e_1-e_m), 
\text{where $\be = (e_1,\ldots,e_m)$}\\
&= (-y_2,-y_3,\ldots,-y_{m-1},-y_1), 
\text{where $\by = (y_1,\ldots,y_m)$.}
\end{align}
\end{subequations}

We are now ready for our main result. 

\begin{theorem}{\bf (the geometry conjecture is true)}
\label{t:main}
The vector $\bv$ defined in \cref{e:defbv} (see also \cref{e:anotherbv}) is 
the sought-after difference vector (see \cref{ss:geocon}).
\end{theorem}
\begin{proof}
We must verify \cref{e:Fm}. 

First, let $z_m\in F_m$.
Then $z_m$ is the $m$th component of some cycle $\bz$.
Obviously, $\bz \in \bC$.
By \cref{e:main1}, there exists $x\in X$ such that 
\begin{equation}
z_1=e_1+x,z_2=e_2+x,\ldots,z_{m-1}=e_{m-1}+x,z_m = e_m+x.
\end{equation}
Hence 
\begin{subequations}
\begin{align}
z_{m}&\in C_m\\
z_{m}&=e_m+x = (e_{m-1}+x) + (e_m-e_{m-1}) = z_{m-1}+v_{m-1}\in C_{m-1}+v_{m-1}\\
z_{m}&=(e_{m-2}+x)+(e_{m-1}-e_{m-2})+(e_{m}-e_{m-1})\in C_{m-2}+v_{m-2}+v_{m-1}\\
&\;\;\vdots \\
z_{m}&\in C_1 + v_1+v_2+\cdots +v_{m-1}.
\end{align}
\end{subequations}
We deduce that 
\begin{equation}
F_m \subseteq C_m\cap(C_{m-1}+v_{m-1})\cap\cdots \cap (C_1 + v_1 + \cdots +v_{m-1}).
\end{equation}

We now tackle the converse inclusion. 
Let 
\begin{equation}
c_m\in C_m\cap(C_{m-1}+v_{m-1})\cap\cdots \cap (C_1 + v_1 + \cdots +v_{m-1}).
\end{equation}
So there exist $c_1\in C_1,\ldots,c_m\in C_m$ such that 
\begin{subequations}
\begin{align}
c_m 
&= c_{m-1}+v_{m-1}\\
&= c_{m-2}+v_{m-2}+v_{m-1}\\
&\;\; \vdots \\
&= c_2 + v_2 + \cdots + v_{m-1}\\
&= c_1 + v_1 + v_2 + \cdots + v_{m-1}.
\end{align}
\end{subequations}
It follows that 
$c_{m-1}=c_{m-2}+v_{m-2}$,
$\ldots$,
$c_{2}=c_1+v_1$,
and 
$c_1=c_m+v_m$
(because $c_m=c_1+v_1+v_2+\cdots v_{m-1} = c_1-v_m$). 
Setting $\bc := (c_1,\ldots,c_m)$, we rewrite this as 
$\bc = \bR\bc + \bR\bv$. 
Using \cref{e:defbv},
$\bc = \bR\bc +\bR(\bR^*\be-\be)
= \bR\bc + \be-\bR\be$.
Hence 
$(\Id-\bR)(\bc-\be)=0$ 
and thus $\bc-\be \in \ker(\Id-\bR) =\bD$. 
It follows that $\bc-\be \in \bD\cap (\bC-\be)$
and now \cref{e:main1} yields 
\begin{equation}
  \bc = \be + (\bc-\be)\in \be+ \big(\bD\cap(\bC-\be)\big) = \bZ.
\end{equation}
Therefore, 
\begin{equation}
  c_m \in F_m
\end{equation}
which completes the proof of the geometry conjecture!
\end{proof}

\section{The case when $m=2$}
\label{s:m=2}
Throughout this section, we assume that $m=2$.

\subsection{Revisiting known results}

It is instructive to revisit this case even if we know the answer already;
moreover, we will discover a new formula for the difference vector $\bv$.
By \cref{e:defofR} and \cref{e:defofT}, 
$\bR^*=\bR$ and $\bT = 0$.
Hence \cref{e:200705a} turns into 
$0 \in \thalb\by + \partial\sigma_{\bC+\bD}(\by)$
$\Leftrightarrow$
$0\in\thalb\by + N_{\overline{\bC+\bD}}^{-1}(\by)$
$\Leftrightarrow$
$-\thalb\by \in N_{\overline{\bC+\bD}}^{-1}(\by)$
$\Leftrightarrow$
$\by \in N_{\overline{\bC+\bD}}(-\thalb\by)$
$\Leftrightarrow$
$\thalb\by \in (\Id+ N_{\overline{\bC+\bD}})(-\thalb\by)$
$\Leftrightarrow$
$-\thalb\by = P_{\overline{\bC+\bD}}(\thalb\by)$
$\Leftrightarrow$
[$-\thalb\by\in \overline{\bC+\bD}$ and 
$(\forall (\bc,\bd)\in\bC\times\bD)$ 
$\tscal{\bc+\bd-(-\thalb\by)}{\thalb\by-(-\thalb\by)}\leq 0$
]
$\Leftrightarrow$
[$\by\in 2\overline{(\bD-\bC)}$ and 
$(\forall (\bc,\bd)\in\bC\times\bD)$ 
$\tscal{2(\bd-\bc)-\by)}{0-\by}\leq 0$
]
$\Leftrightarrow$
\begin{equation}
\label{e:200805e}
  \by = P_{2\overline{(\bD-\bC)}}(0).
\end{equation}
By \cref{e:defofe},
\begin{equation}
\label{e:200805a}
  \be = -\thalb\by - 0\by = -\thalb P_{2\overline{(\bD-\bC)}}(0)
  = P_{-\overline{(\bD-\bC)}}(0)
  = P_{\overline{\bC+\bD}}(0).
\end{equation}
Finally, by \cref{e:defbv}, 
\begin{equation}
\label{e:200805c}
  \bv = -\bR^*\by = -\bR\by.
\end{equation}
We now express these quantities also in 
the underlying space $X$. 
We claim that 
\begin{equation}
\by \overset{?}{=} \big(P_{\overline{C_2-C_1}}(0),P_{\overline{C_1-C_2}}(0) \big)
{=} \big(P_{\overline{C_2-C_1}}(0),-P_{\overline{C_2-C_1}}(0) \big).
\end{equation}
Set $y := P_{\overline{C_2-C_1}}(0)$.
Then $y\leftarrow c_{2,n}-c_{1,n}$, where 
$(c_{1,n},c_{2,n})_\nnn$ is a sequence in $C_1\times C_2$. 
Now for every $\nnn$, 
$c_{2,n}-c_{1,n} = 2(\thalb(c_{1,n}+c_{2,n})-c_{1,n})$
and $c_{1,n}-c_{2,n} = 2(\thalb(c_{1,n}+c_{2,n})-c_{2,n})$,
so 
\begin{equation}
  \big(c_{2,n}-c_{1,n} ,c_{1,n}-c_{2,n} \big) \in 2(\bD-\bC)
\end{equation}
which implies 
$(y,-y)\in 2\overline{(\bD-\bC)}$. 
Next, let us take 
$(\bc,\bd)\in \bC\times\bD$, say 
$\bc = (c_1,c_2)\in C_1\times C_2$ and 
$\bd = (x,x)$ for some $x\in X$. 
Then 
\begin{subequations}
\begin{align}
\scal{2(\bd-\bc)-(y,-y)}{\bzero-(y,-y)}
&= 
\scal{2(x-c_1,x-c_2)-(y,-y)}{(-y,y)}\\
&= 
\scal{(2x-2c_1-y,2x-2c_2+y)}{(-y,y)}\\
&= 
\scal{2x-2c_1-y}{-y} + \scal{2x-2c_2+y}{y}\\
&= 
\scal{2c_2-2c_1-2y}{-y}\\
&= 
2\scal{(c_2-c_1)-y}{0-y}\\
&\leq 0
\end{align}
\end{subequations}
by definition of $y$. 
We have verified
\begin{equation}
\label{e:200805b}
\by = \big(P_{\overline{C_2-C_1}}(0),P_{\overline{C_1-C_2}}(0) \big). 
\end{equation}
It follows (by \cref{e:200805a} and \cref{e:200805b}) that 
\begin{equation}
\label{e:200805b+}
\be = -\thalb\by = -\thalb \big(P_{\overline{C_2-C_1}}(0),P_{\overline{C_1-C_2}}(0) \big) 
= \thalb \big(P_{\overline{C_1-C_2}}(0),P_{\overline{C_2-C_1}}(0) \big) 
\end{equation}
and (by \cref{e:200805c} and \cref{e:200805b})
\begin{equation}
\label{e:200805f}
\bv = -\bR\by = -\big(P_{\overline{C_1-C_2}}(0),P_{\overline{C_2-C_1}}(0) \big)
= \big(P_{\overline{C_2-C_1}}(0),P_{\overline{C_1-C_2}}(0) \big) = \by. 
\end{equation}
Hence 
\begin{equation}
  v_1 = P_{\overline{C_2-C_1}}(0)
  \;\;\text{and}\;\;
  v_2 = P_{\overline{C_1-C_2}}(0) = -v_1
\end{equation}
and this is completely consistent with the known theory 
exposed in \cref{ss:geocon} (see \cref{e:200805d})! 
Along our journey, we have thus discovered a new identity 
for $\bv$ by combining \cref{e:200805e} with \cref{e:200805f} 
which we record in the following result. 

\begin{proposition} $\bv = \big(P_{\overline{C_2-C_1}}(0),P_{\overline{C_1-C_2}}(0)\big) 
  = P_{2\overline{\bD-\bC}}(0) = 2P_{\overline{\bD-\bC}}(0)$.
\end{proposition}

\subsection{Two lines}

\label{s:2lines}
It is instructive to 
consider two general lines in $X$, given by 
\begin{equation}
  C_1 = c_1+\RR u_1,
  \; C_2 = c_2 + \RR u_2,
  \quad
  \text{where}
  \;\;
  c_1\perp u_1,
  \;
  c_2\perp u_2,
  \;
  \text{and}
  \; 
  \|u_1\|=\|u_2\|=1
\end{equation}
because we will obtain descriptions of $\bZ$, $\bv$, $\by$, and $\be$.
We start by noting that for every $i\in\{1,2\}$,
\begin{equation}
(\forall x\in X)\;\;  P_i(x) = c_i + \scal{u_i}{x}u_i.
\end{equation}
Let $\bz = (z_1,z_2) \in C_1 \times C_2$.
Then $z_1 = c_1+\rho_1u_1$ and $z_2 = c_2+\rho_2u_2$
for some $\rho_1,\rho_2$ in $\RR$.
Now assume that $\bz$ is actually a cycle. 
Then $z_2 = P_2P_1z_2$, i.e., 
\begin{subequations}
\begin{align}
c_2+\rho_2u_2
&= z_2 \\
&= P_2P_1z_2 \\
&= c_2+\scal{u_2}{P_1z_2}u_2\\
&= c_2+\scal{u_2}{c_1+\scal{u_1}{z_2}u_1}u_2\\
&= c_2+\big(\scal{u_2}{c_1}+\scal{u_1}{z_2}\scal{u_2}{u_1}\big)u_2\\
&= c_2+\big(\scal{u_2}{c_1}+\scal{u_1}{c_2+\rho_2u_2}\scal{u_2}{u_1}\big)u_2\\
&= c_2+\big(\scal{u_2}{c_1}+ \scal{u_1}{c_2}\scal{u_2}{u_1} + 
\rho_2 \scal{u_1}{u_2}\scal{u_2}{u_1}\big)u_2;
\end{align}
\end{subequations}
equivalently, 
\begin{equation}
\label{e:200808a}
\rho_2\big(1-\scal{u_1}{u_2}^2\big)
= \scal{c_1}{u_2} + \scal{u_1}{u_2}\scal{u_1}{c_2}.
\end{equation}

The theory bifurcates from here as we will see in the following subsections.

\subsubsection{The lines are parallel}

Let's first assume that 
the two lines $C_1,C_2$ are parallel; equivalently, 
$\scal{u_1}{u_2}^2 = 1$. 
Without loss of generality, $u_2 = u_1 =: u$. 
Then \emph{every} $\rho_2$ in $\RR$ solves \cref{e:200808a}. 
It then follows that 
the set of cycles is
\begin{equation}
\bZ = (c_1,c_2)+\RR(u,u). 
\end{equation}
Moreover, using \cref{e:cheatcycle}, \cref{e:200805f}, and 
\cref{e:200805b+}, we obtain 
\begin{equation}
\bv = (c_2-c_1,c_1-c_2) = \by
\;\;\text{and}\;\;
\be = \thalb(c_1-c_2,c_2-c_1).
\end{equation}

\subsubsection{The lines are not parallel}

Now we assume that $C_1,C_2$ are not parallel.
Then $\scal{u_1}{u_2}^2 < 1$ and solving 
\cref{e:200808a} for $\rho_2$ yields
\begin{equation}
\rho_2 := \frac{\scal{u_2}{c_1} + \scal{u_1}{u_2}\scal{u_1}{c_2}}{1-\scal{u_1}{u_2}^2}
\end{equation}
and analogously 
\begin{equation}
\rho_1 := \frac{\scal{u_1}{c_2} + \scal{u_1}{u_2}\scal{u_2}{c_1}}{1-\scal{u_1}{u_2}^2}. 
\end{equation}
Hence the set of cycles $\bZ$ has only one element, namely 
\begin{equation}
  \bz = (z_1,z_2) = (c_1+\rho_1u_1,c_2+\rho_2u_2);
\end{equation}
and
$\bv = (z_2-z_1,z_1-z_2) = \by$ and $\be = -\thalb \bv$
which we don't expand as the expressions don't simplify.

\section{The case when $m=3$}

\label{s:m=3}

Throughout this section, we assume that $m=3$.
Then the matrix representations for $\bT$ (see \cref{e:defofT}) is 
\begin{equation}
\label{e:coldday0}
\bT = \frac{1}{6}\begin{pmatrix}
0 & 0 & 1 \\
1 & 0 & 0 \\
0 & 1 & 0
\end{pmatrix}
-\frac{1}{6} \begin{pmatrix}
0 & 1 & 0 \\
0 & 0 & 1 \\
1 & 0 & 0 
\end{pmatrix}
= \frac{1}{6} \begin{pmatrix*}[r]
0 & -1 & 1 \\
1 & 0 & -1 \\
-1 & 1 & 0
\end{pmatrix*} 
\end{equation}
and thus 
\begin{equation}
\label{e:coldday1}
-\thalb\Id-\bT 
=  \frac{1}{6} \begin{pmatrix*}[r]
  -3 & 1 & -1 \\
  -1 & -3 & 1 \\
  1 & -1 & -3 
  \end{pmatrix*}. 
\end{equation}

Thanks to \cref{e:coldday1}, \cref{p:verifyby}, \cref{e:defofe}, 
and \cref{e:anotherbv},
we obtain the following result:

\begin{theorem}
\label{t:verifyby3}
Let $\by = (y_1,y_2,y_3)\in \bX=X^3$.
Then $\by$ is the unique solution of \cref{e:200705a} if 
and only if all of the following hold: 
\begin{equation}
\label{e:verifyby3a}
y_1+y_2+y_3=0,
\end{equation}
there exist 
sequences $(c_{1,n})_\nnn$ in $C_1$, 
$(c_{2,n})_\nnn$ in $C_2$, 
$(c_{3,n})_\nnn$ in $C_3$,
and $(x_n)_\nnn$ in $X$ such that 
\begin{subequations}
\label{e:verifyby3b}
\begin{align}
 c_{1,n}+x_n &\to \tfrac{1}{6}\big(-3y_1+y_2 - y_3\big),  \\
 c_{2,n}+x_n &\to \tfrac{1}{6}\big(-y_1-3y_2+y_3\big), \\
 c_{3,n}+x_n &\to \tfrac{1}{6}\big(y_1-y_2 - 3y_3\big), 
\end{align}
\end{subequations}
and 
$(\forall (c_1,c_2,c_3)\in C_1\times C_2\times C_3)$
\begin{equation}
\label{e:verifyby3c}
\scal{y_1}{c_1}+\scal{y_2}{c_2}+\scal{y_3}{c_3} 
\leq -\thalb\big(\|y_1\|^2 + \|y_2\|^2 + \|y_3\|^2 \big).
\end{equation}
If $\by=(y_1,y_2,y_3)$ satisfies all these conditions, then 
\begin{equation}
\label{e:verifyby3e}
\be = (e_1,e_2,e_3) = \tfrac{1}{6}\big(-3y_1+y_2-y_3,-y_1-3y_2+y_3,y_1-y_3-3y_3 \big)
\end{equation}
and 
\begin{equation}
\label{e:verifyby3v}
\bv = -(y_2,y_3,y_1)
\end{equation}
are the vectors from \cref{e:defofe} and \cref{e:defbv}, respectively. 
\end{theorem}

Note that 
if $\bv=(v_1,v_2,v_3)$,
then we can obtain $\by$ through \cref{e:yfromv}:
\begin{equation}
\label{e:3yfromv}
\by = -\bR\bv = -(v_3,v_1,v_2).
\end{equation}
Moreover, if desired, we can find $\be$
by combining \cref{e:defofe} and \cref{e:coldday0}.

\subsection{Three lines}

Let us consider three lines, which can be treated 
similar to two lines (see \cref{s:2lines}).
(For brevity, we will omit full details on 
the somewhat tedious algebraic manipulations.)
We assume that 
\begin{equation}
  C_1 = c_1+\RR u_1,
  \; C_2 = c_2 + \RR u_2,
  \; C_3 = c_3 + \RR u_3,
\end{equation}
where
\begin{equation}
  c_1\perp u_1,
  \;
  c_2\perp u_2,
  \;
  c_3\perp u_3\;\;
  \text{and}
  \;\;
  \|u_1\|=\|u_2\|=\|u_3\|=1. 
\end{equation}

\subsubsection{All three lines are parallel}

\label{sss:para}
Let's first assume that all lines are parallel;
equivalently, 
$\scal{u_3}{u_2}\scal{u_2}{u_1}\scal{u_1}{u_3} = 1$.
Without loss of generality, $u := u_1=u_2=u_3$. 
Then the set of cycles is 
\begin{equation}
  \bZ = (c_1,c_2,c_3) + \RR(u,u,u)
\end{equation}
and thus the difference vector is 
\begin{equation}
\bv = (c_2-c_1,c_3-c_2,c_1-c_3).
\end{equation}
In \cref{fig:2}, we visualize this case for three lines in $\RR^3$.

\begin{figure}
\centering
\includegraphics[width=0.95\linewidth]{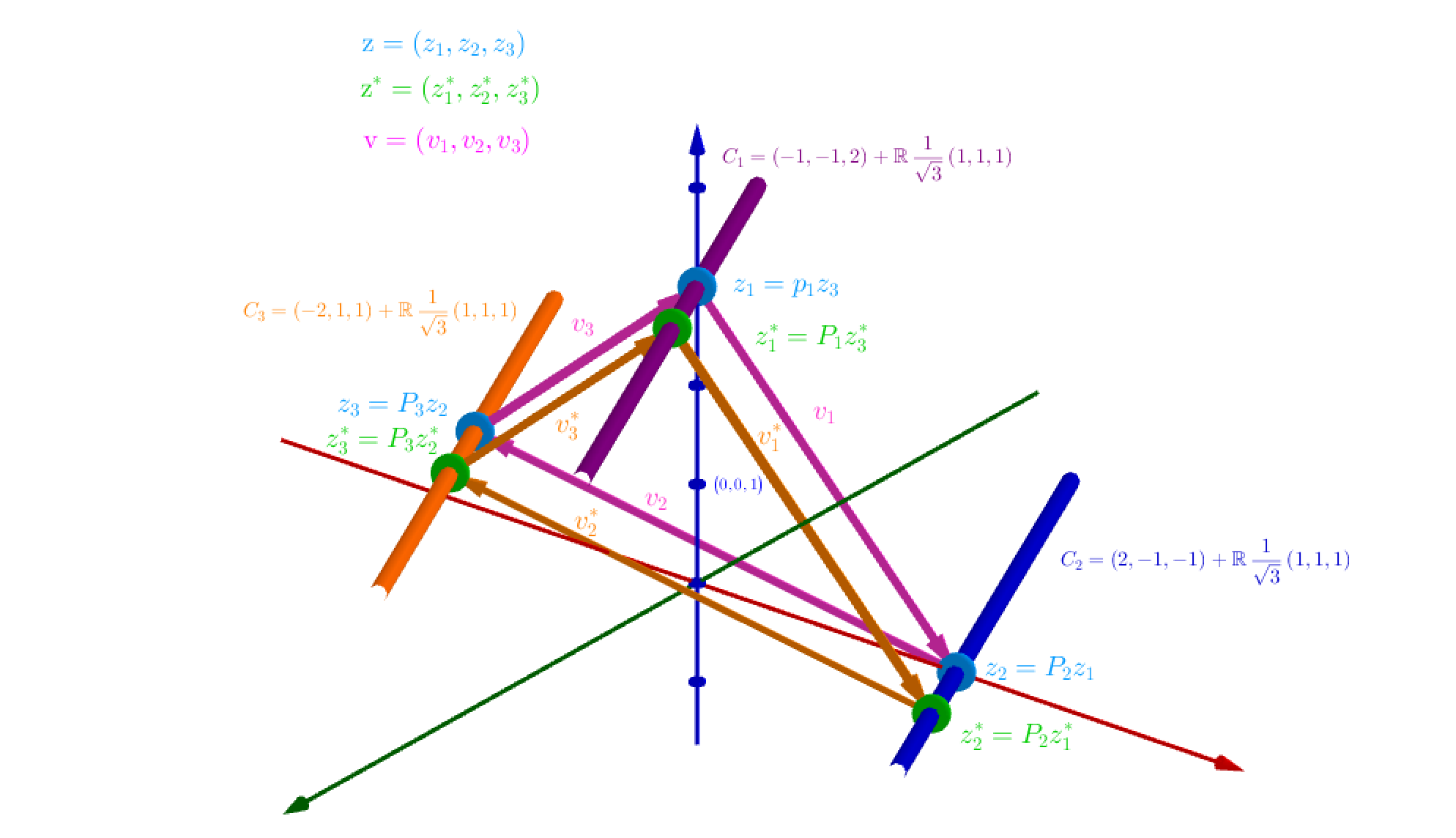}
\caption{Visualization of the cycles and 
the difference vectors for three parallel lines in $\RR^3$. 
See \cref{sss:para} for details.}
\label{fig:2}
\end{figure}

\subsubsection{Not all three lines are parallel}

\label{sss:notpara}

The case when not all lines are parallel corresponds to 
$\scal{u_3}{u_2}\scal{u_2}{u_1}\scal{u_1}{u_3} < 1$.
Then the set of cycles $\bZ$ consists is a singleton containing 
\begin{equation}
\bz = (z_1,z_2,z_3) = (c_1+\rho_1u_1,c_2+\rho_2u_2,c_3+\rho_3u_3), 
\end{equation}
where 
\begin{subequations}
\begin{align}
  \rho_1 &:=
  \frac{\scal{u_1}{c_3}+\scal{u_1}{u_3}\scal{u_3}{c_2}%
  +\scal{u_1}{u_3}\scal{u_3}{u_2}\scal{u_2}{c_1}}%
  {1-\scal{u_3}{u_2}\scal{u_2}{u_1}\scal{u_1}{u_3}},\\[+2mm]
  \rho_2 &:=
  \frac{\scal{u_2}{c_1}+\scal{u_2}{u_1}\scal{u_1}{c_3}%
  +\scal{u_2}{u_1}\scal{u_1}{u_3}\scal{u_3}{c_2}}%
  {1-\scal{u_3}{u_2}\scal{u_2}{u_1}\scal{u_1}{u_3}},\\[+2mm]
  \rho_3 &:=
  \frac{\scal{u_3}{c_2}+\scal{u_3}{u_2}\scal{u_2}{c_1}%
  +\scal{u_3}{u_2}\scal{u_2}{u_1}\scal{u_1}{c_3}}%
  {1-\scal{u_3}{u_2}\scal{u_2}{u_1}\scal{u_1}{u_3}},
\end{align}
\end{subequations}
and 
\begin{equation}
  \bv = (z_2-z_1,z_3-z_2,z_1-z_3).
\end{equation}
In \cref{fig:3}, we visualize this case for three lines in $\RR^3$. 
\begin{figure}
\centering
\includegraphics[width=0.95\linewidth]{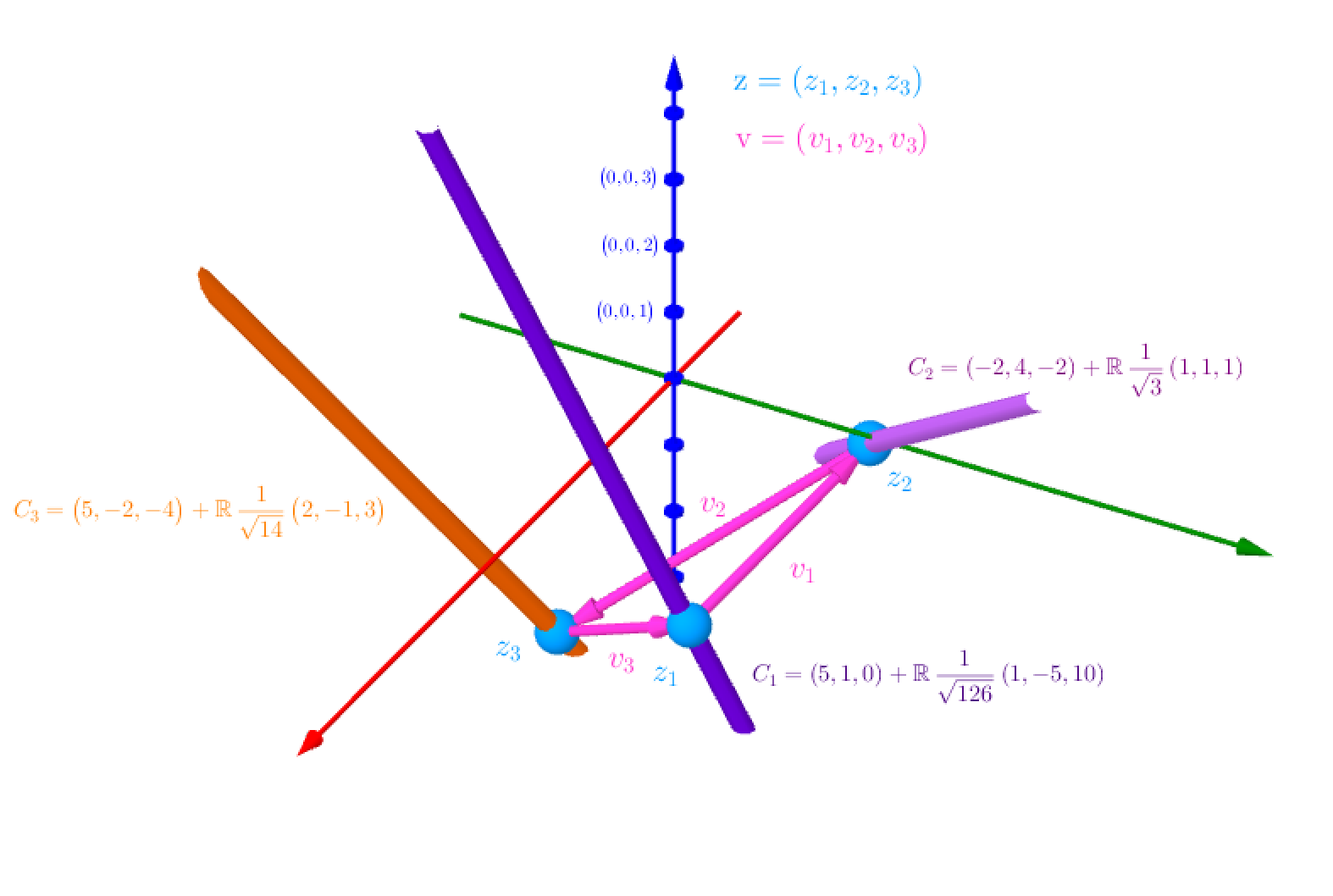}
\caption{Visualization of the cycle and 
the difference vectors for three lines in $\RR^3$ that are not parallel. 
See \cref{sss:notpara} for details.}
\label{fig:3}
\end{figure}

\subsection{An example featuring the epigraph of the exponential function}

\label{ss:expo}
In this section, we specialize further to 
\begin{equation}
X=\RR^2. 
\end{equation}
Inspired by \cite[Section~3]{DeP}, we will present three sets 
in the Euclidean plane and consider two different orderings.
The sets are the epigraph of the exponential function,
\begin{equation}
  \epi(\exp) = \menge{(\xi,\eta)\in\RR^2}{\exp(\xi)\leq\eta} = \gra(\exp)+(\{0\}\times\RR),
\end{equation}
along with the two horizontal lines 
\begin{equation}
  \RR\times\{0\} 
  \;\;\text{and}\;\;
  \RR\times\{1\}. 
\end{equation}

In the following we describe two orderings, one leading to the presence 
of cycles, the other to their absence. 
The case when there are cycles is depicted in \cref{fig:1}. 

\begin{figure}
	\centering
	\includegraphics[width=0.95\linewidth]{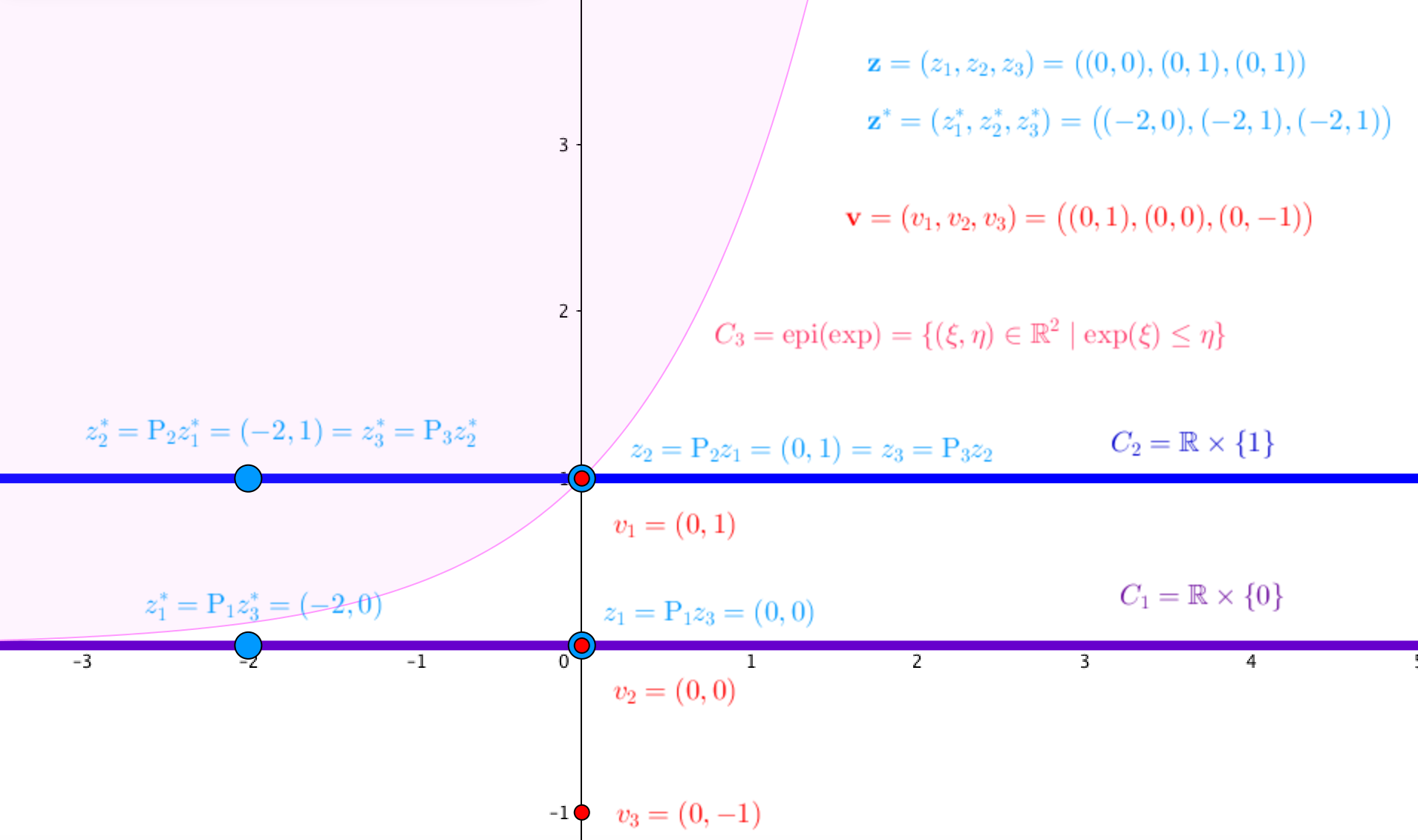}
  \caption{Visualization of the case of two lines and the 
  epigraph of the exponential function where there are cycles. 
  See \cref{ss:expo} for details.}
	\label{fig:1}
\end{figure}

\subsubsection{An ordering with cycles}

In this section, we assume that 
\begin{equation}
C_1 = \RR\times\{0\},\;\;
C_2 = \RR\times\{1\},\;\;
C_3 = \epi(\exp).
\end{equation}
Now set 
\begin{equation}
\label{e:magicy1}
  \by = (y_1,y_2,y_3) = \big((0,1),(0,-1),(0,0)\big). 
\end{equation}
We claim that \cref{e:magicy1} satisfies the characterization 
provided by \cref{t:verifyby3}.

Clearly, $y_1+y_2+y_3=(0+0+0,1-1+0) = (0,0)$ and so \cref{e:verifyby3a} holds. 

Next, set 
$c_{1,n}\equiv (0,0)\in C_1$, 
$c_{2,n}\equiv (0,1)\in C_2$,
$c_{3,n}\equiv (0,1) = (0,\exp(0))\in C_3$, and 
$x_n\equiv (0,-\tfrac{2}{3})\in X$. 
Then
\begin{subequations}
\begin{align}
  c_{1,n} + x_n &\equiv (0,-\tfrac{2}{3}) = \tfrac{1}{6}\big(-3(0,1)+(0,-1)-(0,0)\big)
  = \tfrac{1}{6}\big(-3y_1+y_2-y_3\big),\\
  c_{2,n} + x_n &\equiv (0,\tfrac{1}{3}) = \tfrac{1}{6}\big(-(0,1)-3(0,-1)+(0,0)\big)
  = \tfrac{1}{6}\big(-y_1-3y_2+y_3\big),\\
  c_{3,n} + x_n &\equiv (0,\tfrac{1}{3}) = \tfrac{1}{6}\big((0,1)-(0,-1)-3(0,0)\big)
  = \tfrac{1}{6}\big(y_1-y_2-3y_3\big),
\end{align}
\end{subequations}
and thus \cref{e:verifyby3b} holds. 

Now let $c_1=(\gamma_1,0)\in C_1$,
$c_2=(\gamma_2,1)\in C_2$, and 
$c_3 = (\gamma_3,\exp(\gamma_3)+\delta_3)\in C_3$, 
where $\{\gamma_1,\gamma_2,\gamma_3\}\subseteq \RR$, and $\delta_3\in\RP$. 
Then
\begin{subequations}
\begin{align}
  &\hspace{-1cm}\scal{y_1}{c_1}+\scal{y_2}{c_2}+\scal{y_3}{c_3} \\
  &= \scal{(0,1)}{(\gamma_1,0)}+\scal{(0,-1)}{(\gamma_2,1)}
  +\scal{(0,0)}{(\gamma_3,\exp(\gamma_3)+\delta_3)}\\
  &= -1\\
  &= -\thalb\big(1+1 +0\big)\\
  &= -\thalb\big(\|(0,1)\|^2 + \|(0,-1)\|^2 + \|(0,0)\|^2 \big)\\
  &= -\thalb\big(\|y_1\|^2 + \|y_2\|^2 + \|y_3\|^2 \big).
\end{align}
\end{subequations}
and thus \cref{e:verifyby3c} holds (even with equality). 

Next, using \cref{e:verifyby3e} and \cref{e:verifyby3v}, 
we obtain 
\begin{subequations}
\begin{align}
  \be &= (e_1,e_2,e_3) = \big((0,-\tfrac{2}{3}),(0,\tfrac{1}{3}),(0,\tfrac{1}{3}) \big)\\
  \bv &= (v_1,v_2,v_3) = \big((0,1),(0,0),(0,-1)\big). 
\end{align}
\end{subequations}
The vector $\bv$ allows us to find the fixed point sets $F_1,F_2,F_3$ 
(see \cref{e:defF_i}) via \cref{t:main}. 
For instance,
\begin{align}
  F_3 &= C_3 \cap (C_2+v_2) \cap (C_1+v_1+v_2)\\
  &=\epi(\exp)\cap(\RR\times\{1\}+(0,0))\cap (\RR\times\{0\}+(0,1))\\
  &=\epi(\exp)\cap(\RR\times\{1\})\cap(\RR\times\{1\})\\
  &= \RM\times\{1\},
\end{align}
which can also be seen geometrically.

\subsubsection{An ordering without cycles}

\label{sss:nocycles}

In this section, we assume that 
\begin{equation}
C_1 = \RR\times\{1\},\;\;
C_2 = \RR\times\{0\},\;\;
C_3 = \epi(\exp),
\end{equation}
which is nearly the same set up as in the last --- the crucial difference
is that $C_1$ and $C_2$ were interchanged!
Now set 
\begin{equation}
\label{e:magicy2}
  \by = (y_1,y_2,y_3) = \big((0,-1),(0,1),(0,0)\big). 
\end{equation}
We claim that \cref{e:magicy2} satisfies the characterization 
provided by \cref{t:verifyby3}.

Clearly, $y_1+y_2+y_3=(0+0+0,-1+1+0) = (0,0)$ and so \cref{e:verifyby3a} holds.

Next, set $(\forall\nnn)$
$c_{1,n}= (-n,1)\in C_1$, 
$c_{2,n}= (-n,0)\in C_2$,
$c_{3,n}= (-n,\exp(-n)\in C_3$, and 
$x_n= (n,-\tfrac{1}{3})\in X$. 
Then
\begin{subequations}
\begin{align}
  c_{1,n} + x_n &\equiv (0,\tfrac{2}{3}) = \tfrac{1}{6}\big(-3(0,-1)+(0,1)-(0,0)\big)
  = \tfrac{1}{6}\big(-3y_1+y_2-y_3\big),\\
  c_{2,n} + x_n &\equiv (0,-\tfrac{1}{3}) = \tfrac{1}{6}\big(-(0,-1)-3(0,1)+(0,0)\big)
  = \tfrac{1}{6}\big(-y_1-3y_2+y_3\big),\\
  c_{3,n} + x_n &= (0,\exp(-n)-\tfrac{1}{3})\notag\\
  &\to (0,-\tfrac{1}{3}) = \tfrac{1}{6}\big((0,-1)-(0,1)-3(0,0)\big)
  = \tfrac{1}{6}\big(y_1-y_2-3y_3\big),
\end{align}
\end{subequations}
and thus \cref{e:verifyby3b} holds. 

Now let $c_1=(\gamma_1,1)\in C_1$,
$c_2=(\gamma_2,0)\in C_2$, and 
$c_3 = (\gamma_3,\exp(\gamma_3)+\delta_3)\in C_3$, 
where $\{\gamma_1,\gamma_2,\gamma_3\}\subseteq \RR$, and $\delta_3\in\RP$. 
Then
\begin{subequations}
\begin{align}
  &\hspace{-1cm}\scal{y_1}{c_1}+\scal{y_2}{c_2}+\scal{y_3}{c_3} \\
  &= \scal{(0,-1)}{(\gamma_1,1)}+\scal{(0,1)}{(\gamma_2,0)}
  +\scal{(0,0)}{(\gamma_3,\exp(\gamma_3)+\delta_3)}\\
  &= -1\\
  &= -\thalb\big(1+1 +0\big)\\
  &= -\thalb\big(\|(0,-1)\|^2 + \|(0,1)\|^2 + \|(0,0)\|^2 \big)\\
  &= -\thalb\big(\|y_1\|^2 + \|y_2\|^2 + \|y_3\|^2 \big).
\end{align}
\end{subequations}
and thus \cref{e:verifyby3c} holds (again with equality). 

Next, using \cref{e:verifyby3e} and \cref{e:verifyby3v}, 
we obtain 
\begin{subequations}
\begin{align}
  \be &= (e_1,e_2,e_3) = \big((0,\tfrac{2}{3}),(0,-\tfrac{1}{3}),(0,-\tfrac{1}{3}) \big)\\
  \bv &= (v_1,v_2,v_3) = \big((0,-1),(0,0),(0,1)\big). 
\end{align}
\end{subequations}
The vector $\bv$ allows us to find the fixed point sets $F_1,F_2,F_3$ 
(see \cref{e:defF_i}) via \cref{t:main}. 
For instance,
\begin{align}
  F_3 &= C_3 \cap (C_2+v_2) \cap (C_1+v_1+v_2)\\
  &=\epi(\exp)\cap(\RR\times\{0\}+(0,0))\cap (\RR\times\{1\}+(0,-1))\\
  &=\epi(\exp)\cap(\RR\times\{0\})\cap(\RR\times\{0\})\\
  &= \varnothing, 
\end{align}
which again can also be seen geometrically. 

\section{Finding the difference vectors for $m\leq 5$ by Banach}

\label{s:Banach}

In this section, we discuss an iterative technique 
to compute $\by$ (given by \cref{e:200705a}) which can be used 
to obtain the difference vector $\bv$ via \cref{e:anotherbv}. 
Note that \cref{e:200705a} is equivalent to 
\begin{equation}
\label{e:200708a}
-\thalb\by - \bT\by \in N_{\overline{\bC+\bD}}^{-1}(\by).
\end{equation}
In this section, let us abbreviate 
\begin{empheq}[box=\mybluebox]{equation}
\label{e:defofbP}
\bP := P_{\overline{\bC+\bD}}, 
\end{empheq}
which is a projector and 
hence \emph{firmly nonexpansive}.
It follows that $\Id-\bP$ is also firmly nonexpansive, 
hence \emph{nonexpansive} ($1$-Lipschitz continuous). 
This allows us to 
rewrite \cref{e:200708a} as 
$\by \in N_{\overline{\bC+\bD}}(-\thalb\by-\bT\by)$
$\Leftrightarrow$
$\by + (-\thalb\by-\bT\by) \in 
(\Id + N_{\overline{\bC+\bD}})(-\thalb\by-\bT\by)$
$\Leftrightarrow$
$-\thalb\by-\bT\by = \bP(\thalb\by-\bT\by)$
$\Leftrightarrow$
$(\thalb\by-\bT\by)-\by=\bP(\thalb\by-\bT\by)$
$\Leftrightarrow$
\begin{equation}
\label{e:200805g}
  (\Id-\bP)(\thalb\by-\bT\by) = \by.
\end{equation}
Because we know already that $\Id-\bP$ is nonexpansive, 
we can solve \cref{e:200805g} by the Banach contraction principle 
as long as 
the inner operator
\begin{equation}
\label{e:200805h}
  \thalb\Id - \bT
\end{equation}
is a nice Banach contraction, i.e., Lipschitz continuous with a constant
strictly less than $1$!
We can determine the operator norm of \cref{e:200805h} by 
analyzing the corresponding matrix in $\RR^{m\times m}$. 
Recall that the singular values are the square roots of 
the (necessarily nonnegative) eigenvalues of 
the symmetric matrix associated with $(\thalb\Id-\bT)^*(\thalb\Id-\bT)$. 
The operator norm is the largest singular value. 
All this can be found using a symbolic algebra package such as 
\texttt{SageMath} (or \texttt{Maple} or \texttt{Mathematica}); 
see \cref{tab:table1} which provides the squared singular values (with multiplicity)
as well as the desired operator norm.

\renewcommand{\arraystretch}{1.5}
\begin{table}[h!]
  \begin{center}
    \begin{tabular}{lll}
      \toprule 
      $m$ & eigenvalues of $(\thalb\Id-\bT)^*(\thalb\Id-\bT)$ & $\|\thalb\Id - \bT\|$\\
      \midrule 
      2 & $\tfrac{1}{4}$ (twice)  & $\tfrac{1}{2} = 0.5$\\
      3 & $\tfrac{1}{3}$ (twice), $\tfrac{1}{4}$  & $\tfrac{1}{\sqrt{3}}\approx 0.58$  \\
      4 & $\tfrac{1}{2}$ (twice), $\tfrac{1}{4}$ (twice) & $\tfrac{1}{\sqrt{2}}\approx 0.71$\\
      5 & $\tfrac{1}{2}+\tfrac{1}{2\sqrt{5}}$ (twice), $\tfrac{1}{2}-\tfrac{1}{2\sqrt{5}}$ (twice), $\tfrac{1}{4}$ 
      & $\sqrt{\tfrac{1}{2}+\tfrac{1}{2\sqrt{5}}}\approx 0.85$ \\
      6 & $1$ (twice), $\tfrac{1}{3}$ (twice), $\tfrac{1}{4}$ (twice) & $1$ \\
      \bottomrule 
    \end{tabular}
    \caption{Computing the operator norm of \cref{e:200805h}.}
    \label{tab:table1}
  \end{center}
\end{table}

\renewcommand{\arraystretch}{1}

Therefore, when $m\leq 5$, then 
the fixed point equation \cref{e:200805g} can theoretically be solved 
by the Banach contraction mapping principle. 
(When $m\geq 7$, the operator norm $\|\thalb\Id - \bT\|$ appears to be always 
strictly larger than $1$.)
Unfortunately, we do not know of an explicit formula for the projector defined 
in \cref{e:defofbP}. In practice, one may appeal to 
\emph{Seeger's algorithm} \cite{Seeger} for computing $\bP$, 
which we record now: 

\begin{fact}{\bf (Seeger's algorithm)}
\label{f:Seeger}
Given 
\begin{subequations}
\begin{equation}
  \bx\in \bX, \;\text{and}\; \bd_0\in\bX, 
\end{equation}
generate sequences 
$(\bc_n)_{n\geq 1}$ and $(\bd_n)_{n\geq 1}$ iteratively via 
\begin{equation}
  \bc_n := P_\bC(\bx-\bd_{n-1}),\quad 
  \bd_n := P_\bD(\bx-\bc_n).
\end{equation}
Then 
\begin{equation}
  \bc_n+\bd_n \to P_{\overline{\bC+\bD}}(\bx) = \bP(\bx).
\end{equation}
\end{subequations}
\end{fact}

\section{Finding the difference vectors by forward-backward}
\label{s:fb} 
In this section, we sketch another approach to numerically
compute the difference vectors. 
We begin by revisiting 
\cref{e:200705a} as a \emph{primal} problem: 

\begin{proposition}
\label{p:gencycle}
We interpret 
\begin{equation}
\label{e:200806a}
  0 \in N_{\overline{\bC+\bD}}^{-1}(\by) + \big(\thalb\Id +\bT\big)(\by), 
\end{equation}
which is \cref{e:200705a} and for which the solution 
$\by$ is unique, as an 
Attouch-Th\'era primal problem 
for the pair $(N_{\overline{\bC+\bD}}^{-1},\thalb\Id+\bT)$.
Then $\thalb\Id+\bT$ is $\thalb$-strongly monotone and 
$(\thalb\Id+\bT)^{-1}$ is $\thalb$-cocoercive.
Moreover, the Attouch-Th\'era dual problem of 
\cref{e:200806a} is 
\begin{equation}
\label{e:200723a}
  0 \in N_{\overline{\bC+\bD}}(\bx) + \big(\thalb\Id +\bT\big)^{-1}(\bx), 
\end{equation}
and the solution set of \cref{e:200723a} is the singleton 
\begin{equation}
\label{e:200806b}
  \{\be\}=\bD^\perp \cap \Fix(P_{\overline{\bC+\bD}}\bR). 
\end{equation}
\end{proposition}
\begin{proof}
Because $\bT$ is skew (see \cref{e:Tskew}), 
it follows that $\thalb\Id+\bT$ is $\thalb$-strongly monotone.
By \cite[Example~22.7]{BC2017}, $(\thalb\Id+\bT)^{-1}$ is $\thalb$-cocoercive. 
Because $\bT$ is linear, the Attouch-Th\'era dual of 
\cref{e:200806a} with respect to the pair $(N_{\overline{\bC+\bD}}^{-1},\thalb\Id+\bT)$
is indeed \cref{e:200723a}.
We can pass from $\by$, the unique solution of \cref{e:200806a}, 
to the set of solutions of \cref{e:200723a} via 
$N_{\overline{\bC+\bD}}^{-1}(\by) \cap -\big(\thalb\Id +\bT\big)(\by)$ 
(see \cite[Proposition~2.4]{BBHM}). 
Because $\bT$ is single-valued, this implies that 
the \emph{unique} solution to \cref{e:200723a} is 
\begin{equation}
\label{e:200723b}
\bx = N_{\overline{\bC+\bD}}^{-1}(\by) \cap -\big(\thalb\Id +\bT\big)(\by)
= -\thalb\by - \bT\by = \be,
\end{equation}
where we used \cref{e:defofe} for the last equality. 
Now consider \cref{e:200723a} again. We rewrite this, 
using \cref{e:200723b}, \cref{e:200715b} and \cref{e:defofe} as 
\begin{equation}
  0 \in N_{\overline{\bC+\bD}}(\be) + \big(\Id-\bR+2P_{\Delta}\big)(\be)
  = N_{\overline{\bC+\bD}}(\be) + (\Id-\bR)(\be),
\end{equation}
or as 
$\be = P_{\overline{\bC+\bD}}(\bR\be) \in \overline{\bC+\bD}$
which yields \cref{e:200806b}. 
\end{proof}

\begin{theorem}
\label{t:fb}
Let $\gamma\in\zeroun$, let $\bx_0\in\bX$, and
generate a sequence $(\bx_n)_\nnn$ via 
\begin{subequations}
\begin{align}
\bx_{n+1} &= 
P_{\overline{\bC+\bD}}\big(\bx_n-\gamma(\thalb\Id+\bT)^{-1}\bx_n\big)\label{e:fb1}\\
&=
P_{\overline{\bC+\bD}}\big((1-\gamma)\bx_n+\gamma\bR\bx_n-2\gamma P_\bD\bx_n\big)\label{e:fb2}.
\end{align}
\end{subequations}
Then 
\begin{equation}
\label{e:fbe}
  \bx_n\to \be,
\end{equation}
\begin{equation}
\label{e:fby}
\bR\bx_n-\bx_n-2P_{\bD}\bx_n \to \by, 
\end{equation}
and 
\begin{equation}
\label{e:fbv}
\bR^*\bx_n-\bx_n\to\bv.
\end{equation}
\end{theorem}
\begin{proof}
Set $\bA := N_{\overline{\bC+\bD}}$.
Also set 
$\bB :=(\thalb\Id+\bT)^{-1}$, which is $\beta$-cocoercive, with $\beta = \thalb$, 
by \cref{p:gencycle}.
Then $\gamma \in \left]0,2\beta\right[$. 
Now set $\delta := 2-\gamma/(2\beta) = 2-\gamma>1$ and 
$\lambda := \lambda_n\equiv 1$.
Then $\lambda_n(\delta-\lambda_n) \equiv \delta-1>0$
and thus $\sum_{n\in\NN}\lambda_n(\delta-\lambda_n)=+\infty$. 
We now apply \cite[Theorem~26.14]{BC2017} on the forward-backward algorithm 
applied to the problem \cref{e:200723a}. 
Note that \cref{e:fb1} is precisely the forward-backward algorithm 
with the parameters just defined because of \cite[Remark~26.15]{BC2017}
and $\bx_{n+1} = J_{\gamma\bA}(\bx_n-\gamma \bB\bx_n)$. 
The alternative formula \cref{e:fb2} follows from \cref{e:200715b}.
Using \cite[Theorem~26.14(i)\&(ii)]{BC2017} and \cref{p:gencycle},
we have $\bx_n\weakly \be$ and 
\begin{equation}
\label{e:fb3}
  \bB\bx_n=\big(\thalb\Id+\bT\big)^{-1}\bx_n \to 
  \big(\thalb\Id+\bT\big)^{-1}\be = -\by.
\end{equation}
(The fact that $\bB\bx_n\to-\by$ and not $\by$ stems from the fact 
that the dual problem in \cite[Chapter~26]{BC2017} differs from the one 
in this paper by a negative sign.)
Now \cref{e:fb3} and \cref{e:200715b} yield \cref{e:fby}. 
Next, \cref{e:fb3} and the fact that $\bT$ is continuous and single-valued yields 
\begin{equation}
\label{e:fb4}
  \bx_n=\big(\thalb\Id+\bT\big)\big(\thalb\Id+\bT\big)^{-1}\bx_n \to 
  \big(\thalb\Id+\bT\big)\big(\thalb\Id+\bT\big)^{-1}\be = \be
\end{equation}
and so \cref{e:fbe} is verified. 
Finally, to check \cref{e:fbv}, apply the continuous operator $\bR^*-\Id$ to \cref{e:fb4} and 
recall \cref{e:defbv}.
\end{proof}


\begin{remark} 
\cref{t:fb} is a powerful result for computing 
$\be,\by,\bv$ as \emph{strong} limits of sequence.
As in \cref{s:Banach}, the numerical difficulty lies 
in the computation of $P_{\overline{\bC+\bD}}$;
however, Seeger's algorithm (see \cref{f:Seeger})
may be used to approximate this projection.
\end{remark}

\begin{remark}
\cref{t:fb} allows for flexibility of the parameter 
$\gamma\in\zeroun$.
Perhaps the most natural choice is 
\begin{equation}
  \gamma = \frac{1}{2};
\end{equation}
however, let us point out an intriguing other choices, namely 
\begin{equation}
  \gamma = \frac{m}{m+2}.
\end{equation}
With the latter choice and \cref{e:JulianPbD}, the inner (forward) operator in 
\cref{e:fb2} 
turns into 
\begin{subequations}
\label{e:crazy}
\begin{align}
(1-\gamma)\Id + \gamma\bR-2\gamma P_\bD
&= 
\frac{2}{m+2}\Id + \frac{m}{m+2}\bR -\frac{2m}{m+2}P_\bD
\\ 
&= 
\frac{2}{m+2}\Id + \frac{m}{m+2}\bR -\frac{2m}{m+2}\frac{1}{m}\sum_{k=0}^{m-1}\bR^k\\
&= 
\frac{m-2}{m+2}\bR
-\frac{2}{m+2}\sum_{k=2}^{m-1}\bR^k,
\end{align}
which is Lipschitz continuous with constant 
$3(m-2)/(m+2)$ because $\bR$ is an isometry. 
\end{subequations}
We point out the cases when $m=2$ and $m=3$,
for which $\gamma=1/2$ and $\gamma=3/5$ respectively, 
and \cref{e:crazy} turns into 
\begin{equation}
\label{e:crazy2}
\text{\rm \big[$m=2$ and $\gamma=\thalb$\big]}
\;\;\Rightarrow\;\;
(1-\gamma)\Id + \gamma\bR-2\gamma P_\bD \equiv 0
\text{\rm ~is $0$-Lipschitz }
\end{equation}
and 
\begin{equation}
\label{e:crazy3}
\text{\rm \big[$m=3$ and $\gamma=\tfrac{3}{5}$\big]}
\;\;\Rightarrow\;\;
(1-\gamma)\Id + \gamma\bR-2\gamma P_\bD = \tfrac{1}{5}\bR-\tfrac{2}{5}\bR^2
\text{\rm ~is $\tfrac{3}{5}$-Lipschitz.}
\end{equation}
Note that \cref{e:crazy2} looks at first puzzling because 
then \cref{e:fb2} turns into 
$\bx_{n+1} = P_{\overline{\bC+\bD}}(0)$
and so \cref{e:fbe} yields $\be = P_{\overline{\bC+\bD}}(0)$;
however, we already observed this directly in \cref{e:200805a}. 
\end{remark}

\section{Conclusion and future work}

\label{s:done}

Using the framework of monotone operator theory,
we resolved the geometry conjecture completely. 
We obtained alternative descriptions of the set of cycles $\bZ$. 
We also sketched numerical approaches for the computation 
of the difference vector $\bv$ by using Seeger's algorithm.

Turning to future research, it is desirable to 
devise algorithms for computing $\bv$ without having to employ
Seeger's algorithm. Moreover, it is interesting to extend 
the results in this paper from projectors to (underrelaxed) projectors 
or even proximal mappings. We have taken steps in this direction, 
and initial progress appears to be quite promising \cite{ABWprox}.

\section*{Acknowledgments}
The authors thank two referees and the editor for 
careful reading, thoughtful comments, and suggestions 
which significantly improved the presentation of the results.
HHB and XW are supported by the Natural Sciences and
Engineering Research Council of Canada.


\small

\begin{thebibliography}{999}

  \seppfour

\bibitem{ABRW} 
S.\ Alwadani, H.H.\ Bauschke, J.P.\ Revalski, and X.\ Wang,
Resolvents and Yosida approximations of displacement mappings of 
isometries, 
\emph{Set-Valued and Variational Analysis}, in press. 
\url{https://link.springer.com/article/10.1007%2Fs11228-021-00584-2} 

\bibitem{ABWprox} 
S.\ Alwadani, H.H.\ Bauschke, and X.\ Wang,
Attouch-Th\'era duality, generalized cycles and gap vectors, 
\emph{SIAM Journal on Optimization}, in press. 

\bibitem{AtTh} 
H.\ Attouch and M.\ Th\'era,
A general duality principle for the sum of two operators,
\emph{Journal of Convex Analysis}~3 (1996), 1--24. 


\bibitem{BCC} 
J.-B.\ Baillon, P.L.\ Combettes, and R.\ Cominetti,
There is no variational characterization of the cycles 
in the method of periodic projections,
\emph{Journal of Functional Analysis}~262 (2012), 400--408.

\bibitem{BCC+} 
J.-B.\ Baillon, P.L.\ Combettes, and R.\ Cominetti,
Asymptotic behavior of compositions of underrelaxed 
nonxpansive operators, 
\emph{Journal of Dynamics and Games}~1 (2014), 331--346.



\bibitem{02.pdf} 
H.H.\ Bauschke and J.M.\ Borwein,
Dykstra's alternating projection algorithm for two sets,
\emph{Journal of Approximation Theory}~79 (1994), 418--443.

\bibitem{01.pdf} 
H.H.\ Bauschke and J.M.\ Borwein,
On the convergence of von Neumann's alternating projection 
algorithm for two sets,
\emph{Set-Valued Analysis}~1 (1993), 185--212. 

\bibitem{BBL} 
H.H.\ Bauschke, J.M.\ Borwein, and A.S.\ Lewis,
The method of cyclic projections for closed convex 
sets in Hilbert space,
\emph{Contemporary Mathematics}~104 (1997), 1--38.

\bibitem{BBHM} 
H.H.\ Bauschke, R.I.\ Bo\c{t}, W.L.\ Hare, and W.M.\ Moursi,
Attouch-Th\'era duality revisited: paramonotonicity and operator 
splitting, 
\emph{Journal of Approximation Theory}~164 (2012), 1065--1084. 

\bibitem{BC2017} 
H.H.\ Bauschke and P.L.\ Combettes,
\emph{Convex Analysis and Monotone Operator Theory in Hilbert Spaces},
second edition,
Springer, 2017.


\bibitem{Victoria} 
H.H.\ Bauschke, V.\ Mart\'in-M\'arquez, S.M.\ Moffat, and X.\ Wang,
Compositions and convex combinations of asymptotically regular 
firmly nonexpansive mappings are also asymptotically regular,
\emph{Fixed Point Theory and Applications}~2012:53. 
\url{https://fixedpointtheoryandapplications.springeropen.com/articles/10.1186/1687-1812-2012-53}





\bibitem{Brezis} 
H.\ Br\'ezis,
\emph{Operateurs Maximaux Monotones et
Semi-Groupes de Contractions dans les Espaces de Hilbert},
North-Holland/Elsevier, 1973. 



\bibitem{BurIus} 
R.S.\ Burachik and A.N.\ Iusem,
\emph{Set-Valued Mappings and Enlargements
of Monotone Operators},
Springer-Verlag, 2008.


\bibitem{ByrneSP}
C.L.\ Byrne,
\emph{Signal Processing}, 
A K Peters, 2005. 

\bibitem{ByrneAIM}
C.L.\ Byrne,
\emph{Applied Iterative Methods},
A K Peters, 2008.

\bibitem{CenZak} 
Y.\ Censor and M.\ Zaknoon,
Algorithms and convergence results of projection methods for 
inconsistent feasibility problems: A review,
\emph{Pure and Applied Functional Analysis}~3 (2018), 565--586.
\url{http://www.ybook.co.jp/online2/oppafa/vol3/p565.html}

\bibitem{CG} 
W.\ Cheney and A.A.\ Goldstein,
Proximity maps for convex sets, 
\emph{Proceedings of the AMS}~10 (1959), 448--450. 

\bibitem{ComPes20}
P.L.\ Combettes and J.-C.\ Pesquet,
Fixed point strategies in data science,
\emph{IEEE Transactions on Signal Processing}, in press. 

\bibitem{DeP} 
A.R.\ De Pierro, 
From parallel to sequential projection methods and
vice versa in convex feasibility: 
results and conjectures, 
in 
\emph{Inherently Parallel Algorithms in Feasibility and Optimization 
and their Applications}, 
D.\ Butnariu, Y.\ Censor and S.\ Reich (editors), 
Elsevier, 2001, pp.~369--379.

\bibitem{CRW} 
R.\ Cominetti, V.\ Roshchina, and A.\ Williamson,
A counterexample to De Pierro's conjecture on the 
convergence of under-relaxed cyclic projections,
\emph{Optimization}~68 (2019), 3--12. 






\bibitem{Iusem} 
A.N.\ Iusem,
On some properties of paramonotone operators,
\emph{Journal of Convex Analysis}~5 (1998), 269--278.









\bibitem{Seeger} 
A.\ Seeger,
Alternating projection and decomposition with respect to two convex sets, 
\emph{Mathematica Japonica}~47 (1998), 273--280.
See also the preprint version 
\url{https://mathfiles.kfupm.edu.sa/data/files/mathonly/TechnicalReportsData/172.pdf}

\bibitem{Simons1}
S.\ Simons,
\emph{Minimax and Monotonicity},
Springer-Verlag,
1998.

\bibitem{Simons2}
S.\ Simons,
\emph{From Hahn-Banach to Monotonicity},
Springer-Verlag,
2008.





\end{thebibliography}
\end{document}